\documentclass[a4paper,10pt]{amsart}
\usepackage[left=1.33in,right=1.33in,top=1.1in,bottom=1.1in,includehead,includefoot]{geometry}
\usepackage[T1]{fontenc}
\usepackage{lmodern,microtype}
\usepackage{amsthm,amsmath,amssymb,colonequals}
\usepackage{graphicx}
\usepackage{enumitem}
\usepackage[pdftitle={On ordered Ramsey numbers of bounded-degree graphs},
            pdfauthor={M. Balko, V. Jelinek, P. Valtr}]{hyperref}

\newtheorem{theorem}{Theorem}[section]
\newtheorem{lemma}[theorem]{Lemma}
\newtheorem{corollary}[theorem]{Corollary}

\newtheorem{problem}[theorem]{Problem}

\newtheorem*{claim*}{Claim}

\DeclareMathOperator{\R}{R} 
\DeclareMathOperator{\OR}{\overline{R}} 
\DeclareMathOperator{\minr}{\min\overline{R}} 
\DeclareMathOperator{\maxr}{\max\overline{R}} 

\newcommand{\cG}{\mathcal{G}}
\newcommand{\cH}{\mathcal{H}}
\newcommand{\cK}{\mathcal{K}}
\newcommand{\cM}{\mathcal{M}}

\newcommand{\cR}{\mathcal{R}}
\newcommand{\cP}{\mathcal{P}}
\newcommand{\cC}{\mathcal{C}}
\newcommand{\cT}{\mathcal{T}}
\newcommand{\cF}{\mathcal{F}}

\makeatletter
\let\old@setaddresses\@setaddresses
\def\@setaddresses{\bgroup\parindent 0pt\let\scshape\relax\old@setaddresses\egroup}
\makeatother

\title[On ordered Ramsey numbers of bounded-degree graphs]{On ordered Ramsey numbers of \\bounded-degree graphs}

\author{Martin Balko}
\author{V\'{\i}t Jel\'{\i}nek}
\author{Pavel Valtr}

\address[Martin Balko, Pavel Valtr]{Department of Applied Mathematics, Faculty of Mathematics and Physics, Charles University, Prague, Czech Republic}
\email{balko@kam.mff.cuni.cz}
\address[V\'{\i}t Jel\'{\i}nek]{Computer Science Institute of Charles University, Faculty of Mathematics and Physics, Charles University, Prague, Czech Republic}
\email{jelinek@iuuk.mff.cuni.cz}

\begin{document}

\begin{abstract}
An \emph{ordered graph} is a pair $\mathcal{G}=(G,\prec)$ where $G$ is a graph and $\prec$ is a total ordering of its vertices.
The \emph{ordered Ramsey number} $\OR(\mathcal{G})$ is the minimum number $N$ such that every $2$-coloring of the edges of the ordered complete graph on $N$ vertices contains a monochromatic copy of $\mathcal{G}$. 

We show that for every integer $d \geq 3$, almost every $d$-regular graph $G$ satisfies $\OR(\cG) \geq \frac{n^{3/2-1/d}}{4\log{n}\log{\log{n}}}$ for every ordering $\cG$ of $G$.
In particular, there are 3-regular graphs $G$ on $n$ vertices for which the numbers $\OR(\cG)$ are superlinear in $n$, regardless of the ordering $\cG$ of $G$.
This solves a problem of Conlon, Fox, Lee, and Sudakov.

On the other hand, we prove that every graph $G$ on $n$ vertices with maximum degree 2 admits an ordering $\cG$ of $G$ such that $\OR(\cG)$ is linear in $n$.

We also show that almost every ordered matching $\cM$ with $n$ vertices and with interval chromatic number two satisfies $\OR(\cM) \geq cn^2/\log^2{n}$ for some absolute constant $c$.
\end{abstract}

\maketitle

\section{Introduction}
\label{sec:introduction}

Let $G$ be a \emph{graph} and let $V(G)$ and $E(G)$ be the set of vertices and edges of~$G$, respectively.
The \emph{Ramsey number} of a graph $G$, denoted by $\R(G)$, is the minimum integer $N$ such that every 2-coloring of the edges of the complete graph $K_N$ on $N$ vertices contains a monochromatic copy of $G$ as a subgraph.
By Ramsey's theorem~\cite{rams30}, the number $\R(G)$ is finite for every graph $G$.

An \emph{ordered graph} is a pair $\cG=(G,\prec)$ where $\prec$ is a total ordering of $V(G)$.
We call $\cG$ an \emph{ordering} of $G$.
The \emph{vertices} of $\cG$ are vertices of $G$ and, similarly, the $\emph{edges}$ of $\cG$ are edges of $G$.
We say that two ordered graphs $\cG=(G,\prec_1)$ and $\cH=(H,\prec_2)$ are \emph{isomorphic} if the graphs $G$ and $H$ are isomorphic via a one-to-one correspondence $g \colon V(G) \to V(H)$ that preserves the orderings.
That is, we have $u\prec_1 v \Leftrightarrow g(u)\prec_2 g(v)$ for all $u,v \in V(G)$.
Note that, for every positive integer $n$, there is only one ordered complete graph on $n$ vertices up to isomorphism.
We use $\cK_n$ to denote such ordered complete graph.
An ordered graph $\cG=(G,\prec_1)$ is an \emph{ordered subgraph} of an ordered graph $\cH=(H,\prec_2)$, written $\cG \subseteq\cH$, if $G$ is a subgraph of $H$ and $\prec_1$ is a suborder of $\prec_2$. 

Recently, an analogue of Ramsey numbers for general ordered graphs has been introduced~\cite{bckk13,bckk15,cfls14} (see also Section~3.7 in~\cite{cfs15}).
The \emph{ordered Ramsey number} $\OR(\cG)$ of $\cG$ is the minimum integer $N$ such that every $2$-coloring of the edges of~$\cK_N$ contains a monochromatic ordered subgraph that is isomorphic to $\cG$.
In this paper, we use the shorter term $\emph{coloring}$ of an ordered graph $\cG=(G,\prec)$ to denote a $2$-coloring of the edges of $G$. 

Obviously, we have $\R(G)\leq \OR(\cG)$ for every ordering $\cG$ of $G$ and $\R(K_n)=\OR(\cK_n)$ for every positive integer $n$.
Since every ordered graph $\cG$ on $n$ vertices is an ordered subgraph of $\cK_n$, we have $\OR(\cG)\leq \R(\cK_n)$.
It follows that the number $\OR(\cG)$ is finite for every ordered graph $\cG$.
A classical result of Erd\H{o}s~\cite{erdos47} gives $\R(K_n) \leq 2^{2n}$.
Thus the number $\OR(\cG)$ is at most exponential in $n$ for every ordered graph $\cG$ on $n$ vertices.

In the following, we omit the ceiling and floor signs whenever they are not crucial.
All logarithms in this paper are base 2.

\section{Previous results}
\label{subsec:previousResults}

For a graph $G$ and its ordering $\cG$, there might be a substantial difference between $\R(G)$ and $\OR(\cG)$.
For a graph $G$, we define
$$\minr(G):=\min\{ \OR(\cG): \cG\ \text{is an ordering of}\ G\},$$
$$\maxr(G):=\max\{ \OR(\cG): \cG\ \text{is an ordering of}\ G\}.$$

By the definitions, for every graph $G$ on $n$ vertices and for every  ordering $\cG$ of $G$,
$$ R(G) \le \minr(G)\le\OR(\cG)\le\maxr(G)\le\OR(\cK_n)=R(K_n).$$

In the 1980s, Chv\'{a}tal, R\"{o}dl, Szemer\'{e}di, and Trotter~\cite{crst83} showed that the Ramsey number $\R(G)$ of every $n$-vertex graph $G$ of constant maximum degree is linear in $n$.
In contrast, Balko, Cibulka, Kr\'{a}l, and Kyn\v{c}l~\cite{bckk13} and, independently, Conlon, Fox, Lee, and Sudakov~\cite{cfls14} recently showed that there are ordered matchings $\cM_n$ on $n$ vertices with $\OR(\cM_n)$ superpolynomial in $n$.

\begin{theorem}[\cite{bckk13,cfls14}]
\label{thm-superpolynomialMatching}
There is a constant $C>0$ such that for every even $n\ge4$, 
\[\maxr(M_n) \geq n^{\log{n}/C\log{\log{n}}},\]
where $M_n$ is a perfect matching (that is, a $1$-regular graph) on $n$ vertices.
\end{theorem}

In fact, Conlon~et~al.~\cite{cfls14} showed that this lower bound is true for almost every ordered matching on $n$ vertices.

A subset $I$ of vertices of an ordered graph $\cG=(G,\prec)$ is an \emph{interval} if  for every pair $u,v$ of vertices of $I$ with $u\prec v$, every vertex $w$ of $G$ satisfying $u\prec w\prec v$ is contained in $I$.
The following notion is strongly related to the growth rate of ordered Ramsey numbers, as we will see later.
The \emph{interval chromatic number} of an ordered graph $\mathcal{G}$ is the minimum number of intervals the vertex set of $\mathcal{G}$ can be partitioned into such that there is no edge between vertices of the same interval.

We say that a sequence of pairwise disjoint intervals $(I_1,\ldots,I_m)$ in $\cG$ is \emph{order-obeying}, if $u \prec v$ for all $u \in I_i$, $v \in I_{i+1}$, and $i\in [m-1]$.
That is, $I_i$ is to the \emph{left} of $I_{i+1}$ for every $i \in [m-1]$.

Balko~et~al.~\cite{bckk13} and Conlon~et~al.~\cite{cfls14} independently showed that for every $n$-vertex ordered graph $\cG$ of bounded degeneracy and bounded interval chromatic number the ordered Ramsey number $\OR(\cG)$ is at most polynomial in the number of vertices of $\cG$.
The bound by  Conlon~et~al.~\cite{cfls14} is stronger and we state it as the following theorem.

\begin{theorem}[\cite{cfls14}]
\label{thm-intChromNum}
Every $d$-degenerate ordered graph $\cG$ with~$n$ vertices and with the interval chromatic number $\chi$ satisfies
\[\OR(\cG) \leq n^{32d\log{\chi}}.\]
\end{theorem}

The following result of Conlon~et~al.~\cite{cfls14} shows that restricting the interval chromatic number does not necessarily force the ordered Ramsey numbers to be linear in the number of vertices.

\begin{theorem}[\cite{cfls14}]
\label{thm-lowerBoundMatch}
There is a constant $C>0$ such that for all $n$ there is an ordered matching $\cM_{2n}$ of
interval chromatic number two with $2n$ vertices satisfying
\[\OR(\cM_{2n}) \geq \frac{Cn^2}{\log^2{n}\log{\log{n}}}.\]
\end{theorem}

Conlon~et~al.~\cite{cfls14} proved this result using the well-known \emph{Van der Corput sequence}.
Some applications of this sequence appear in the discrepancy theory~\cite{mat99}.

Using an elementary argument (see~\cite{cfls14}), it can be shown that every ordered matching $\cM_{2n}$ with $2n$ vertices and with interval chromatic number two satisfies $\OR(\cM_{2n}) \leq 2n^2$.
Asymptotically, this is the best known upper bound on $\OR(\cM_{2n})$.
It would be interesting to close this gap~\cite[Problem~6.3]{cfls14}.

By the result of Conlon~et~al.~\cite{cfls14} mentioned below Theorem~\ref{thm-superpolynomialMatching},
the ordered Ramsey number of almost every ordered matching is superpolynomial in the number of vertices.
In contrast, it is not difficult to find an ordered matching~$\cM$ on $2n$ vertices with $\OR(\cM)$ only linear in $n$.
An example is the ordered matching $\mathcal{M}_{2n}$ on $[2n]\colonequals\{1,\ldots,2n\}$ with edges $\{i,2n+1-i\}$ for $i=1,\ldots,n$;
see part~a) of Figure~\ref{fig-linearOrderings}.
Another example is the ordered matching $\cM'_{2n}$ sketched in part~b) of Figure~\ref{fig-linearOrderings}.
The simple fact that $\OR(\cM_{2n}), \OR(\cM'_{2n}) \leq 4n-2$ can be proved easily using the pigeonhole principle.

\begin{figure}[ht]
\centering
\includegraphics{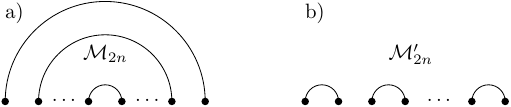}
\caption{Two orderings of 1-regular graphs with the ordered Ramsey number linear in the number of vertices.}
\label{fig-linearOrderings}
\end{figure}

In particular, it follows that every graph $G$ with maximum degree 1 satisfies $\minr(G)\le4n-2$.
It is a natural question whether for every fixed positive integer $\Delta$ there is a constant $c_\Delta$ such that
$\minr(G)\le c_\Delta n$ holds for every graph $G$ on $n$ vertices with maximum degree $\Delta$.
Conlon~et~al.~\cite{cfls14} consider this to be unlikely and pose the following problem.

\begin{problem}[{\cite[Problem~6.7]{cfls14}}]
\label{prob-2Regular}
Do random 3-regular graphs have superlinear ordered Ramsey numbers for all orderings?
\end{problem}

\section{Our results}
\label{subsec:ouResults}

As our first main result, we give an affirmative answer to Problem~\ref{prob-2Regular}. In fact, we solve the problem in a slightly more general setting, by extending the concept of $d$-regular graphs to non-integral values of~$d$.
For a real number $\rho>0$ and a positive integer $n$ with $\lceil \rho n \rceil$ even, we say that a graph $G$ on $n$ vertices is \emph{$\rho$-regular}, if every vertex of $G$ has degree $\lfloor \rho\rfloor$ or $\lceil \rho\rceil$ and the total number of edges of $G$ is $\lceil\rho n\rceil/2$.

Note that when $\rho$ is an integer, the above definition coincides exactly with the standard notion of regular graphs.
We let $G(\rho,n)$ denote the random $\rho$-regular graph on $n$ vertices drawn uniformly and independently from the set of all $\rho$-regular graphs on the vertex set $[n]$.

The graph $G(\rho,n)$ satisfies an event $A$ \emph{asymptotically almost surely} if the probability that $A$ holds tends to 1 as $n$ goes to infinity.

We may now state our first main result. Part~\ref{item1-thmLowerBound} in it shows that random $\rho$-regular graphs have superlinear minimum ordered numbers for any fixed real number $\rho>2$,
and part~\ref{item2-thmLowerBound} in it shows that there are actually ``almost 2-regular'' graphs with superlinear minimum ordered Ramsey numbers.

\begin{theorem}
\label{thm-lowerBoundAllOrderings}
\leavevmode
\begin{enumerate}[label=(\roman*)]
\item\label{item1-thmLowerBound} For every fixed real number $\rho>2$, asymptotically almost surely
\[\minr(G(\rho,n)) \geq \frac{n^{3/2-1/\rho}}{4\log{n}\log{\log{n}}}.\]
In particular, almost every 3-regular $n$-vertex graph $G$ satisfies 
\[\minr(G) \geq n^{7/6}/(4\log{n}\log{\log{n}}),\]

\item\label{item2-thmLowerBound} Asymptotically almost surely,
\[\minr\left(G\left(2+\frac{9\log{\log{n}}}{\log{n}},n\right)\right) \geq \frac{n\log{n}}{2\log{\log{n}}}.\]
\end{enumerate}
\end{theorem}

In contrast to this result, we show that the trivial linear lower bound is asympto\-tically the best possible for graphs of maximum degree two.

\begin{theorem}
\label{thm-2Regular}
There is an absolute constant $C$ such that for every graph $G$ on $n$ vertices with maximum degree 2, $\minr(G) \leq Cn$.
\end{theorem}

In fact, the following stronger Tur\'{a}n-type statement is true when $G$ is bipartite.

\begin{theorem}
\label{thm-evenCycleInBlowUp}
For every real $\varepsilon >0$, there is a constant $C(\varepsilon)$ such that, for every integer $n$, every bipartite graph $G$ on $n$ vertices with maximum degree 2 admits an ordering $\cG$ of $G$ that is contained in every ordered graph with $N \colonequals C(\varepsilon) n$ vertices and with at least $\varepsilon N^2$ edges.
\end{theorem}

We note that no such Tur\'{a}n-type statement is true for a $2$-regular graph $G$ that is not bipartite, since then $G$ contains an odd cycle and thus no ordering of such a graph is contained in any ordering of the complete bipartite graph $K_{N/2,N/2}$ with $N^2/4$ edges.

For the multi-colored case, perhaps the following result, which is in similar spirit as Theorem~\ref{thm-evenCycleInBlowUp}, might be true.
For every graph $G$ on $n$ vertices with maximum degree $2$ and for every $\varepsilon > 0$, there is a constant $C(\varepsilon)>0$ such that every ordered graph on $N \colonequals C(\varepsilon) n$ vertices with at least $\varepsilon N^3$ copies of $\mathcal{K}_3$ contains some ordering of $G$ as an ordered subgraph.
If true, such a statement would immediately imply a version of Theorem~\ref{thm-2Regular} for any bounded number of colors.

For the upper bounds in the case of larger maximum degree, a simple corollary of Theorem~\ref{thm-intChromNum}, states that every graph $G$ on $n$ vertices with constant maximum degree $\Delta$ admits an ordering $\cG$ with $\OR(\cG)$ polynomial in~$n$.
To obtain such a bound, it suffices to consider a proper coloring of $V(G)$ with $\Delta+1$ colors and order $G$ by placing the $\Delta+1$ color classes as $\Delta+1$ disjoint intervals.
The interval chromatic number of the resulting ordering is then $\Delta+1$ and we obtain the following bound from Theorem~3.1 in~\cite{cfls14}.

\begin{corollary}
\label{cor-upperBoundAllOrderings}
For a positive integer $\Delta$, every graph $G$ with $n$ vertices and with maximum degree $\Delta$ satisfies $\minr(G) \leq O(n^{(\Delta+1)\lceil\log(\Delta+1)\rceil+1})$.

\end{corollary}

Note that the gap between the upper bounds from Corollary~\ref{cor-upperBoundAllOrderings} and the lower bounds from Theorem~\ref{thm-lowerBoundAllOrderings} is rather large.
It would be interesting to close it at least for $\Delta=3$.

For a positive integer $n$, the \emph{random $n$-permutation} is a permutation of the set $[n]$ chosen independently uniformly at random from the set of all $n!$ permutations of the set $[n]$.

For a positive integer $n$ and the random $n$-permutation $\pi$, the \emph{random ordered $n$-matching $\cM(\pi)$} is the ordered matching with the vertex set $[2n]$ and with edges $\{i,n+\pi(i)\}$ for every $i \in [n]$.
Note that the interval chromatic number of every random ordered $n$-matching is two.

The random ordered $n$-matching satisfies an event $A$ \emph{asymptotically almost surely} if the probability that $A$ holds tends to 1 as $n$ goes to infinity.

As our last main result, we show an improved lower bound on ordered Ramsey numbers of ordered matchings with interval chromatic number two, which is a step towards answering Problem~6.3 in~\cite{cfls14}.
We improve the bound of Conlon~et~al.~\cite{cfls14} from Theorem~\ref{thm-lowerBoundMatch} by eliminating the $(\log{\log{n}})$-factor in the denominator.

\begin{theorem}
\label{thm-improvedLowerBoundMatch}
There is a constant $C>0$ such that the random ordered $n$-matching $\cM(\pi)$ asymptotically almost surely satisfies
\[\OR(\cM(\pi)) \geq C\left(\frac{n}{\log{n}}\right)^2.\]
\end{theorem}

This result has an immediate corollary to a certain Ramsey-type problem for $\{0,1\}$-matrices.
A $\{0,1\}$-matrix $B$ is \emph{contained} in a $\{0,1\}$-matrix $A$ if $B$ can be obtained from $A$ by deleting some rows and columns and by replacing some 1-entries by 0.
For a $\{0,1\}$-matrix $A$, we let $\overline{A}$ denote the matrix obtained from~$A$ by replacing 1-entries by 0 and 0-entries by 1.

Theorem~\ref{thm-improvedLowerBoundMatch} implies that there is an $n \times n$ permutation matrix $P$ and  a matrix $A \in \{0,1\}^{N \times N}$ with $N \colonequals \Omega(n^2/\log^2{n})$ such that neither $A$ nor $\overline{A}$ contains $P$.
Using a similar argument as for ordered matchings~\cite{cfls14}, it is possible to show that every $n \times n$ permutation matrix is contained in $A$ or $\overline{A}$ for every $n^2 \times n^2$ $\{0,1\}$-matrix $A$.

\section{Large ordered Ramsey numbers for all orderings of regular graphs}
\label{sec:proofLowerBoundAllOrderings}

Here we prove Theorem~\ref{thm-lowerBoundAllOrderings} by showing that, asymptotically almost surely, random $\rho$-regular graphs have superlinear ordered Ramsey numbers for all orderings.
The main ingredient of the proof is the following technical result.

\begin{theorem}
\label{thm-lowerBoundAllOrderingsAlmostReg}
Let $\{\varepsilon_n\}_{n \geq 1}$, $\{\zeta_n\}_{n \geq 1}$, and
$\{\rho_n\}_{n \geq 1}$ be sequences of real numbers satisfying these constraints:
\begin{itemize}
 \item $0<\varepsilon_n$ and $0<\zeta_n = o(1)$ for every $n$ large enough,
 \item there is a constant $C$ such that $1\leq\rho_n\leq C$ for all $n$,
 \item $\lim_{n\to\infty} \left((\frac{1}{2}-\varepsilon_n)\rho_n-\zeta_n-1\right)n\log n =\infty$.
\end{itemize}
Then asymptotically almost surely,
\[\minr(G(\rho_n,n)) \geq \frac{\zeta_n n^{1+\varepsilon_n}}{2\log{(1/\zeta_n})}.\]
\end{theorem}

We first show how Theorem~\ref{thm-lowerBoundAllOrderings} follows from Theorem~\ref{thm-lowerBoundAllOrderingsAlmostReg}.
\ref{item1-thmLowerBound}
For every $n \geq 2$, let $\rho_n \colonequals \rho$, $\varepsilon_n \colonequals 1/2-1/\rho-1/\log{n}$, and $\zeta_n\colonequals 1/\log{n}$.
Since $\rho >2$, we have $0<\varepsilon_n$ for $n$ large enough, and the remaining assumptions of Theorem~\ref{thm-lowerBoundAllOrderingsAlmostReg} are satisfied as well.
Part~\ref{item1-thmLowerBound} then follows directly from Theorem~\ref{thm-lowerBoundAllOrderingsAlmostReg}.

\ref{item2-thmLowerBound}
It suffices to set $\rho_n \colonequals 2+\frac{9\log{\log{n}}}{\log{n}}$, $\varepsilon_n \colonequals \frac{2\log{\log{n}}}{\log{n}}$, and $\zeta_n \colonequals \frac{1}{\log{n}}$ and apply Theorem~\ref{thm-lowerBoundAllOrderingsAlmostReg}.
The assumptions of Theorem~\ref{thm-lowerBoundAllOrderingsAlmostReg} are satisfied, since
\begin{align*}
 (1/2-\varepsilon_n)\rho_n-\zeta_n-1&= \frac{9\log\log n}{2\log n}-\frac{4\log\log n}{\log n}-18\left(\frac{\log\log n}{\log n}\right)^2-\frac{1}{\log n}\\
&=\frac{\log\log n}{2\log n} - O\left(\frac{1}{\log n}\right).
\end{align*}
Part~\ref{item2-thmLowerBound} follows.

In the rest of the section, we prove Theorem~\ref{thm-lowerBoundAllOrderingsAlmostReg}.
If $A$ and $B$ are subsets of the vertex set of a graph $G$, then we use $e_G(A,B)$ to denote the number of edges that have one vertex in $A$ and one vertex in $B$.
In particular, $e_G(A,A)$ is the number of edges of the subgraph $G[A]$ of $G$ induced by $A$.
A similar notation, $e_{\cG}(A,B) = e_G(A,B)$, is used for an ordered graph $\cG = (G,\prec)$ and two subsets $A$ and $B$ of its vertices.

\begin{lemma}
\label{lem-setPartitionManyEdges}
Let $M,n,s,t$ be positive integers with $M \leq \binom{t}{2}+t$ and let $\rho>0$ and $\delta > 0$ be real numbers such that $\lceil \rho n\rceil$ is even and
\begin{equation}
\label{eq1}
t^n \cdot \binom{t^2}{M} \cdot \binom{s^2M}{\lceil \rho n \rceil/2}<\delta D
\end{equation}
where $D$ is the number of $\rho$-regular graphs on $[n]$.
Then, with probability at least $1-\delta$, the graph $G(\rho,n)$ satisfies the following statement: for every partition of $V(G(\rho,n))$ into sets $X_1,\ldots,X_t$, each of size at most $s$, there are more than $M$ pairs $(X_i,X_j)$ with $1\leq i \leq j \leq t$ and $e_{G(\rho,n)}(X_i,X_j)>0$. 
\end{lemma}
\begin{proof}
Let $X_1,\ldots,X_t$ be a partition of the set $[n]$ such that $X_i$ contains at most $s$ elements for every $1 \leq i \leq t$.
Let $S$ be a set of $M$ pairs $(X_i,X_j)$ for some $1 \leq i \leq j \leq t$.
We let $h$ denote the number of graphs $H$ on $[n]$ with $\lceil \rho n \rceil/2$ edges such that $e_H(X_i,X_j)=0$ for every pair $(X_i,X_j)$ with $1\leq i \leq j \leq t$ that is not contained in $S$.
Since the size of every $X_i$ is at most $s$, we have 
\[h \leq \binom{s^2M}{\lceil \rho n \rceil/2}.\]
Since every $\rho$-regular graph on $n$ vertices contains $\lceil \rho n \rceil/2$ edges, the probability that there are no edges of $G(\rho,n)$ between $X_i$ and $X_j$ for every $(X_i,X_j) \notin S$ is at most $h/D$.

The number of partitions $X_1,\ldots,X_t$ of $[n]$ is at most $t^n$.
The number of choices for the set $S$ is at most $\binom{\binom{t}{2}+t}{M}\leq\binom{t^2}{M}$.
Altogether, the expected number of partitions of $[n]$ into sets $X_1,\ldots,X_t$ of size at most $s$ with at most $M$ pairs $(X_i,X_j)$, $1 \leq i \leq j \leq t$, that satisfy $e_{G(\rho,n)}(X_i,X_j)>0$ is at most 
\[\frac{t^n}{D} \cdot \binom{t^2}{M} \cdot \binom{s^2M}{\lceil \rho n \rceil/2}.\]
By our assumption this term is bounded from above by $\delta$.
By Markov's inequality, the probability that $G(\rho,n)$ satisfies the statement from the lemma is at least $1-\delta$.
\end{proof}

Bender and Canfield~\cite{benCan78} and independently Wormald~\cite{wor78} showed that the number $D$ of $d$-regular graphs on $n$ vertices satisfies
\[D=(1+o(1))\frac{(dn)!}{2^{\frac{dn}{2}}\left(\frac{dn}{2}\right)!(d!)^n}\exp\left(\frac{1-d^2}{4}\right)\]
for a fixed integer $d$ and a sufficiently large $n$ such that $dn$ is even.
Bender and Canfield~\cite{benCan78} also proved an asymptotic formula for the number of labeled graphs with a given degree sequence.
In the case of $\rho$-regular graphs, their result gives the following estimate.

\begin{corollary}[{\cite[Theorems~1 and~2]{benCan78}}]
\label{cor-numberAlmostRegular}
For a real number $\rho \geq 2$ with $\lceil \rho n \rceil$ even, the number of $\rho$-regular graphs with the vertex set $[n]$, for $n$ sufficiently large, is at least
\[\frac{\lceil \rho n\rceil!}{2^{{\lceil \rho n\rceil}/{2}}\left(\frac{\lceil \rho n\rceil}{2}\right)!(d!)^{n}(d+1)^{\lceil \gamma n \rceil}e^{d^2}},\]
where $d\colonequals\lfloor \rho\rfloor$ and $\gamma\colonequals \rho -d \in[0,1)$. 
\end{corollary}

\begin{proof}[Proof of Theorem~\ref{thm-lowerBoundAllOrderingsAlmostReg}]
Let $\varepsilon_n$, $\zeta_n$, and $\rho_n$ be sequences satisfying the assumptions of the theorem.
Note that these assumptions imply $\varepsilon_n<1/2$ for a sufficiently large $n$.
We may also assume $\zeta_n = \omega(1/n^{\epsilon_n})$, as otherwise the statement is trivial.
We set $m \colonequals \lceil \rho_n n \rceil$.
That is, $m$ is the sum of degrees of every $\rho_n$-regular graph on $n$ vertices.

We set $M \colonequals \zeta_n n$, $s \colonequals n^{\varepsilon_n}$, $t \colonequals \frac{\zeta_n n}{2\log{(1/\zeta_n)}}$, and $\delta \colonequals 2^{(\rho_n(\varepsilon_n-1/2)+1+\zeta_n)n\log{n}}$.
Note that by our assumptions on $\varepsilon_n$, $\zeta_n$, and $\rho_n$, the value of $\delta$ tends to zero as $n$ goes to infinity.
We show that the parameters $\delta, M,n,s,t$ satisfy~\eqref{eq1}.

By Corollary~\ref{cor-numberAlmostRegular}, the number $D$ of $\rho_n$-regular graphs on $n$ vertices is asymptotically at least\[\frac{m!}{2^{\frac{m}{2}}\left(\frac{m}{2}\right)!(d_n!)^n(d_n+1)^{\lceil \gamma_n n \rceil}e^{d_n^2}},\]
where $d_n\colonequals\lfloor \rho_n\rfloor$ and $\gamma_n\colonequals \rho_n-d_n$. 

Recalling that the sequence $d_n$ is bounded, and using the estimate $(k/e)^k\leq k! \leq k^k$ for a positive integer $k$, we obtain that $D$ is at least
\[\frac{\left(\frac{m}{e}\right)^{m}}{2^{m/2}\left(\frac{m}{2}\right)^{m/2}d_n^{d_nn}(d_n+1)^{\lceil \gamma_n n \rceil}e^{d_n^2}} 
> \frac{m^{m/2}}{2^{O(n)}}.\]

By the choice of $M,s,t$ and by $D > \frac{m^{m/2}}{2^{O(n)}}$, the left side of~\eqref{eq1} divided by $D$ is at most
\[\frac{2^{O(n)}}{m^{m/2}}\left(\frac{\zeta_n n}{2\log{(1/\zeta_n)}}\right)^n\cdot \binom{\left(\frac{\zeta_n n}{2\log{(1/\zeta_n)}}\right)^2}{\zeta_n n} \cdot \binom{\zeta_n n^{2\varepsilon_n+1}}{m/2}\]
for $n$ sufficiently large.
Applying the estimates $\zeta_n =o(1)$ and $\binom{a}{b} \leq \left(\frac{ea}{b}\right)^b$, we bound this term from above by
\[\frac{1}{m^{m/2}}\cdot n^n\cdot \left(\frac{e \zeta_n n}{4\log^2{(1/\zeta_n)}}\right)^{\zeta_n n}\cdot\left(\frac{2e \zeta_n n^{2\varepsilon_n+1}}{m}\right)^{m/2}\]
for $n$ sufficiently large.
Using elementary calculations, we see that for $n$ large enough, this is less than
\begin{align*}
n^{-{m}/{2} + n+\zeta_n n+(2\varepsilon_n+1){m}/{2} -{m}/{2}}
&\leq2^{( -\rho_n + 1 +\zeta_n +(\varepsilon_n+1/2)\rho_n)n\log n}\\
&=2^{(\rho_n(\varepsilon_n-1/2)+1+\zeta_n)n\log{n}}\\
&=\delta.
\end{align*}
The inequality follows from $m=\lceil \rho_n n\rceil$ and $\varepsilon_n < 1/2$.
We also have $M \leq \binom{t}{2}+t$ due to $\zeta_n = \omega(1/n^{\epsilon_n})$ and $\epsilon_n<1/2$.
That is, the assumptions of Lemma~\ref{lem-setPartitionManyEdges} are satisfied.
Since $\delta$ tends to zero as $n$ goes to infinity, the statement of the lemma holds asymptotically almost surely for this choice of $\delta$, $M$, $s$, $t$.

Let $\cR$ be the ordered complete graph with loops and with the vertex set $[t]$.
Let $\chi$ be a coloring of $\cR$ that assigns either a red or a blue color to every edge of~$\cR$ independently at random with probability $1/2$.
Let $(I_1,\ldots,I_t)$ be a partition of the vertex set of $\cK_{st}$ into an  order-obeying collection intervals, each of size $s$.
We define the coloring $\chi'$ of $\cK_{st}$ such that the color $\chi'(e)$ of an edge $e$ of $\cK_{st}$ is $\chi(\{i,j\})$ if one vertex of $e$ lies in $I_i$ and the other one in $I_j$.
Note that, since $\cR$ contains loops, every edge of $\cK_{st}$ receives some color via $\chi'$.

By Lemma~\ref{lem-setPartitionManyEdges}, asymptotically almost surely, there are more than $M$ pairs $(X_i,X_j)$ with $1 \leq i \leq j \leq t$ and $e_{G(\rho_n,n)}(X_i,X_j)>0$ in every partition of $V(G(\rho_n,n))$ into sets $X_1,\ldots,X_t$ of size at most $s$.
We show that if $G(\rho_n,n)$ satisfies this condition, then we have $\OR(\cG) \geq st$ for every ordering $\cG$ of $G(\rho_n,n)$.

Let $\cG$ be an arbitrary ordering of $G(\rho_n,n)$.
We show that the probability that there is a red copy of $\cG$ in $\chi'$ is less than $1/2$.
Suppose there is a red copy $\cG_0$ of $\cG$ in $\chi'$.
For $i=1,\ldots,t$, let $J_i \colonequals V(\cG_0) \cap I_i$.
Then $J_1,\ldots,J_t$ induces a partition of the vertices of $\cG$ into $t$ (possibly empty) intervals of size at most $s$.
The number of such partitions of $\cG$ is at most 
\[\binom{n+t-1}{t-1} \leq \left(\frac{e(n+t-1)}{t-1}\right)^{t-1} \leq \left(\frac{3n}{t}\right)^t =\left(\frac{6\log{(1/\zeta_n)}}{\zeta_n}\right)^t< 2^{2t\log{(1/\zeta_n)}},\]
where the last inequality follows from $\zeta_n =o(1)$, as then $6\log{(1/\zeta_n)} < 1/\zeta_n$.

By Lemma~\ref{lem-setPartitionManyEdges}, there are more than $M$ pairs $(J_i,J_j)$ with $1 \leq i \leq j \leq t$ and $e_{\cG_0}(J_i,J_j)>0$.
From the choice of $\chi'$, the red copy $\cG_0$ corresponds to an ordered subgraph $\cH$ of $\cR$ with more than $M$ edges that are all red in $\chi$.
Such ordered graph $\cH$ appears in $\cR$ with probability at most $2^{-M-1}$.

The edges of $\cH$ are determined by the partition $J_1,\ldots,J_t$ of $\cG$.
Thus, by the union bound, the probability that there is a red copy of $\cG$ in $\chi'$ is less than
\[2^{2t\log{(1/\zeta_n)}} \cdot 2^{-M-1} = 2^{\zeta_n n-\zeta_n n-1}=1/2.\]

From symmetry, the probability that there is a blue copy of $\cG$ in $\chi'$ is less than $1/2$.
Thus the probability that $\chi'$ contains no monochromatic copy of $\cG$ is positive.
It follows that there is a coloring of $\cK_{st}$ with no monochromatic copy of $\cG$ and we have $\OR(\cG) \geq  st$.

Since $\cG$ is an arbitrary ordering of $G(\rho_n,n)$, we obtain that asymptotically almost surely the graph $G(\rho_n,n)$ satisfies $\minr(G(\rho_n,n)) \geq st$.
\end{proof}

\section{A linear upper bound for graphs of maximum degree two}
\label{sec:proof2Regular}

In this section, we prove Theorem~\ref{thm-2Regular}, which says that every graph on $n$ vertices with maximum degree 2 can be ordered in such a way
that the corresponding ordered Ramsey number is linear in $n$.
We also prove a stronger Tur\'{a}n-type statement, Theorem~\ref{thm-evenCycleInBlowUp}, for bipartite graphs of maximum degree 2.
Since every graph of maximum degree two is a union of vertex disjoint paths, cycles, and isolated vertices,
it suffices to prove these results for 2-regular graphs.
We make no serious effort to optimize the constants.

Every 2-regular graph $G$ is a union of pairwise vertex disjoint cycles.
Thus a first idea how to prove Theorem~\ref{thm-2Regular} might be to find an ordering of every cycle $C_k$ with the ordered Ramsey number linear in $k$ and then place these ordered cycles next to each other into disjoint intervals.
However, this attempt fails in general, as shown by the following example.

Let $\cG$ be such ordering of a 2-regular graph that consists of the cycle $C_{n/3}$ and $2n/9$ copies of~$C_3$.
Let $N \colonequals (n/3-1)4n/9$ and let $\chi$ be the following coloring of $\cK_N$. 
Partition the vertex set of $\cK_N$ into intervals $I_1,\ldots,I_{4n/9}$ of size $n/3-1$ and color all edges between vertices from the same interval $I_i$ blue.
Then color all remaining edges red.
The coloring $\chi$ contains no blue copy of $\cG$, as the longest cycle $C_{n/3}$ has more vertices than any interval $I_i$.
There is also no red copy of $\cG$ in $\chi$, as no two vertices of any red ordered 3-cycle are in the same interval $I_i$.
Altogether, we have $\OR(\cG) \geq \Omega(n^2)$.

Our approach in the proofs of Theorems~\ref{thm-2Regular} and~\ref{thm-evenCycleInBlowUp} is based on so-called alternating paths.
Let $v_1,\dots,v_n$ be the vertices of the $n$-vertex path $P_n$ in the order as they appear along the path.
The \emph{alternating path} $\cP_n=(P_n,\prec)$ is the ordering of $P_n$ where $v_1 \prec v_3 \prec v_5 \prec \cdots \prec v_n \prec v_{n-1} \prec v_{n-3} \prec \cdots \prec v_2$ for $n$ odd and $v_1 \prec v_3 \prec v_5 \prec \cdots \prec v_{n-1} \prec v_{n} \prec v_{n-2} \prec\cdots \prec v_2$ for $n$ even.
See part a) of Figure~\ref{fig-altPath}.

\begin{figure}[ht]
\centering
\includegraphics{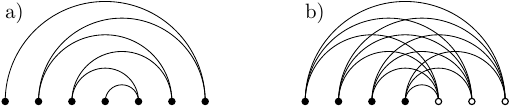}
\caption{a)~The alternating path $\cP_7$. b) The ordered graph $\cK_{4,3}$.}
\label{fig-altPath}
\end{figure}

This ordering of a path is a key ingredient in the proofs of Theorems~\ref{thm-2Regular} and \ref{thm-evenCycleInBlowUp}.
Balko~et~al.~\cite{bckk13} showed that the ordered Ramsey number of the alternating path $\cP_n$ is linear in $n$.
Here we prove the following stronger Tur\'{a}n-type result.

\begin{lemma}
\label{lem-TuranAltPath}
Let $\varepsilon>0$ be a real constant.
Then, for every integer $n$, every ordered graph on $N\geq n/\varepsilon$ vertices with at least $\varepsilon N^2$ edges contains $\cP_n$ as an ordered subgraph.
\end{lemma}
\begin{proof}
For a given $\varepsilon>0$, let $\cH=(H,\prec)$ be an ordered graph on $N\geq n/\varepsilon$ vertices with at least $\varepsilon N^2$ edges.
Without loss of generality, we assume that the vertex set of $\cH$ is $[N]$.

For a vertex $v$ of $\cH$, the \emph{leftmost neighbor of $v$} in $\cH$ is the minimum from $\{u \in V(H) \colon  u \prec v, \{u,v\} \in E(H)\}$, if it exists.
The \emph{rightmost neighbor of $v$} in $\cH$ is the maximum from $\{u \in V(H) \colon v \prec u, \{u,v\} \in E(H)\}$, if it exists.

We consider the following process of removing edges of $\cH$ that proceeds in steps $1,\ldots,n-2$.
In every odd step of the process, we remove edges $\{u,v\}$ for every vertex $v$ of $\cH$ such that $u$ is the leftmost neighbor of $v$ (if it exists).
In every even step, we remove edges $\{u,v\}$ for every vertex $v$ of $\cH$ such that $u$ is the rightmost neighbor of $v$ (if it exists).
Clearly, we remove at most $N$ edges of $\cH$ in every step.
In total, we remove at most $(n-2)N$ edges of $\cH$ once the process is finished.

From the choice of $N$, we have $\varepsilon N^2\geq n N > (n-2)N$ and thus there is at least one edge $\{v_{n-1},v_n\}$ of $\cH$ that we did not remove.
Without loss of generality we assume that $v_n \prec v_{n-1}$ for $n$ odd and $v_{n-1} \prec v_n$ for $n$ even.

We now follow the process of removing the edges of $\cH$ backwards and we construct the alternating path $\cP_n$ on vertices $v_1, \ldots, v_n$.
For $i =n-2,\ldots,1$, we let $v_i$ be the vertex that was removed in the $i$th step of the removing process as a neighbor of $v_{i+1}$.
If there is no such vertex $v_i$, then we would remove the edge $\{v_{i+1},v_{i+2}\}$ in the $i$th step of the removing process, which is impossible.
Consequently, the construction of $\cP_n$ stops at $v_1$ and we obtain an alternating path on $n$ vertices as an ordered subgraph of $\cH$.
\end{proof}

\begin{corollary}
\label{cor-rainbowMatch}
Let $\varepsilon>0$ be a real constant.
Then, for every integer $n$, every ordered graph on $N\geq 2n/\varepsilon$ vertices with at least $\varepsilon N^2$ edges contains $\cM_{2n}$ as an ordered subgraph.
\end{corollary}
\begin{proof}
This follows easily from Lemma~\ref{lem-TuranAltPath}, as $\cM_{2n}$ is an ordered subgraph of~$\cP_{2n}$.
\end{proof}

For positive integers $r$ and $s$, let $\cK_{r,s}$ be the (unique up to isomorphism) ordering of $K_{r,s}$, in which the color classes of $K_{r,s}$ form two disjoint intervals such that the interval of size $r$ is to the left of the interval of size $s$.
See part b) of Figure~\ref{fig-altPath}.

For positive integers $k$ and $n$, we use $\cP^k_n$ to denote the ordered graph that is obtained from the alternating path on $n$ vertices by replacing every edge with a copy of $\mathcal{K}_{k,k}$.
Formally, let $\cP_n$ be the alternating path with the vertex set $[n]$ and let $(B_1,\ldots,B_n)$ be a sequence of $n$ intervals of size $k$ listed from left to right.
Then $\cup_{i=1}^kB_i$ is the vertex set of $\cP^k_n$ and a pair $\{u,v\}$ with $u \in B_i$ and $v \in B_j$ is an edge of $\cP^k_n$ if and only if $\{i,j\}$ is an edge of $\cP_n$.
See Figure~\ref{fig-altBlowUp} for an illustration.
We call the ordered graph $\cP^k_n$ the \emph{$k$-blow-up of $\cP_n$} and the intervals $B_1,\ldots,B_n$ are called the \emph{blocks of $\cP^k_n$}.
Note that $\cP^k_n$ has $kn$ vertices and that $\cP^1_n=\cP_n$.

\begin{figure}[ht]
\centering
\includegraphics{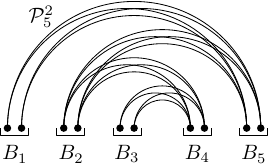}
\caption{The 2-blow-up $\cP^2_5$ of $\cP_5$.}
\label{fig-altBlowUp}
\end{figure}

If $\cH$ is an ordered graph with the vertex set partitioned into an order-obeying sequence of intervals $I_1,\ldots,I_m$, then we say that $\cP^k_n$ is an \emph{ordered subgraph of $\cH$ respecting the partitioning $I_1,\ldots,I_m$} if $\cP^k_n \subseteq \cH$, every block of $\cP^k_n$ is contained in some interval $I_i$, and no two blocks of $\cP^k_n$ are contained in the same interval~$I_i$.

The following result is a variant of Lemma~\ref{lem-TuranAltPath} for $k$-blow-ups of $\cP_n$.

\begin{lemma}
\label{lem-TuranBlowUp}
Let $\varepsilon>0$ be a real constant and let $k$ and $d \geq k(2/\varepsilon)^k$ be positive integers.
Then, for every integer $n$, every ordered graph $\cH$ with $N\geq 2d^{2k+1}\varepsilon^{-1}n$ vertices partitioned into an order-obeying sequence of intervals $I_1,\ldots,I_{N/d}$, each of size $d$, and with at least $\varepsilon N^2$ edges contains $\cP^k_n$ as an ordered subgraph respecting the partitioning $I_1,\ldots,I_{N/d}$.
\end{lemma}
\begin{proof}
For integers $k$ and $m\geq k$, the K\H{o}v\'{a}ri--S\'{o}s--Tur\'{a}n Theorem~\cite{kovari54} says that every bipartite graph with color classes of size $m$, which contains no $K_{k,k}$ as a subgraph, has fewer than $k^{1/k}m^{2-1/k}+km \leq 2k^{1/k}m^{2-1/k}$ edges. 
Since the ordering $\cK_{m,m}$ of $K_{m,m}$ is uniquely determined up to isomorphism, we see that the K\H{o}v\'{a}ri--S\'{o}s--Tur\'{a}n theorem is true in the ordered setting.
That is, every ordered graph, which is contained in $\cK_{m,m}$ and which contains no $\cK_{k,k}$ as an ordered subgraph, has fewer than $2k^{1/k}m^{2-1/k}$ edges.

Let $\cH$ be an ordered graph on $N\geq 2d^{2k+1}\varepsilon^{-1}n$ vertices partitioned into an order-obeying sequence of intervals $(I_1,\ldots,I_{N/d})$, each of size $d$, and with at least $\varepsilon N^2$ edges.

There are at least $\frac{\varepsilon N^2}{2d^2}$ pairs $\{i,j\} \in \binom{[N/d]}{2}$ with $e_\cH(I_i,I_j) \geq 2k^{1/k}d^{2-1/k}$.
Otherwise there are fewer than 
\begin{align*}
\frac{N}{d}\binom{d}{2}&+(1-\varepsilon)\binom{N/d}{2}\cdot2k^{1/k}d^{2-1/k}+\frac{\varepsilon N^2}{2d^2}\cdot d^2 \\
&\leq \frac{N^2}{2d^2}\cdot2k^{1/k}d^{2-1/k}+\frac{\varepsilon N^2}{2d^2}\cdot d^2 =\varepsilon N^2\left(\frac{1}{\varepsilon}\left(\frac{k}{d}\right)^{1/k}+\frac{1}{2}\right)
\leq \varepsilon N^2
\end{align*}
edges in $\cH$, which contradicts our assumptions.
The last inequality follows from $d \geq k(2/\varepsilon)^k$.
By the K\H{o}v\'{a}ri--S\'{o}s--Tur\'{a}n Theorem, there are at least $\frac{\varepsilon N^2}{2d^2}$ pairs $\{I_i,I_j\}$ that induce a copy of $\cK_{k,k}$ in $\cH$.

Let $\cR$ be an ordered graph with the vertex set $[N/d]$ such that $\{i,j\}$ is an edge of $\cR$ if and only if there is a copy of $\cK_{k,k}$ in $\cH$ with one part in $I_i$ and the other one in $I_j$.
For every edge $\{i,j\}$ of $\cR$, we fix one such copy $\cK$ of $\cK_{k,k}$ and say that $\cK$ \emph{represents} the edge $\{i,j\}$.
The \emph{type of the left color class} of $\cK$ is the image of the left color class of $\cK$ via the bijective mapping $I_i \to [d]$ that preserves the ordering of $I_i$.
Similarly, the \emph{type of the right color class} of $\cK$ is the image of the right color class via the bijective mapping $I_j \to [d]$ that preserves the ordering of~$I_j$.

We know that $\cR$ contains at least $\frac{\varepsilon N^2}{2d^2}$ edges.
Let $T$ be a set of all pairs $(A,B)$ of $k$-tuples $A,B\in\binom{[d]}{k}$.
Note that the size of $T$ is $\binom{d}{k}^2\leq d^{2k}$.
We assign a pair $(A,B)$ to every edge $\{i,j\}$ of $\cR$ if the type of the left and the right color class of the copy of $\cK_{k,k}$ that represents $\{i,j\}$ is $A$ and $B$, respectively.
By the pigeonhole principle there are at least $\frac{\varepsilon N^2}{2d^{2k+2}}$ edges of $\cR$ with the same pair $(A_0,B_0)$.

Let $\cR'$ be the ordered subgraph of $\cR$ consisting of edges that were assigned the pair $(A_0,B_0)$.
By our observations, $\cR'$ contains $N/d$ vertices and at least $\frac{\varepsilon}{2d^{2k}}(\frac{N}{d})^2$ edges.
Since $N/d \geq 2d^{2k}\varepsilon^{-1}n$, Lemma~\ref{lem-TuranAltPath} implies that $\cP_n$ is an ordered subgraph of $\cR'$.
This alternating path corresponds to a monochromatic copy of $\cP_n$ in $\cR$.
It follows from the construction of $\cR$ that $\cP^k_n$ is an ordered subgraph of $\cH$ respecting the partitioning $I_1,\ldots,I_{N/d}$.
\end{proof}

We now introduce orderings of cycles that we use in the proofs of Theorems~\ref{thm-2Regular} and~\ref{thm-evenCycleInBlowUp}.
The final ordering of a given 2-regular graph will be obtained by constructing a union of these ordered cycles.
For a positive integer $n$, let $\cP_n=(P_n,\prec)$ be the alternating path on vertices $u_1 \prec \cdots \prec u_n$.

For $n \geq 3$, the even \emph{alternating cycle} $\cC_{2n-2}=(C_{2n-2},<)$ is obtained from $\cP_n$ as follows.
First, for every $i \in [n] \setminus \{1,\lceil \frac{n+1}{2} \rceil\}$, we replace each vertex $u_i$ with two vertices $v_i < w_i$.
For $i\in\{1,\lceil \frac{n+1}{2} \rceil\}$, we set $v_i\colonequals u_i$ and $w_i\colonequals u_i$.
Then, for every edge $\{u_i,u_j\}$ of $\cP_n$, we place edges $\{v_i,w_j\}$ and $\{w_i,v_j\}$ into $\cC_{2n-2}$.
The \emph{blocks of $\cC_{2n-2}$} are the sets $\{v_i,w_i\}$ for $i=1,\ldots,n$.
The edge that contains the $(n-1)$th and the $n$th vertex of $\cC_{2n-2}$ is the \emph{inner edge of $\cC_{2n-2}$}.
See part~a) of Figure~\ref{fig-altCyc}.

For $n \geq 1$, the odd alternating cycle $\cC_{2n+1}=(C_{2n+1},<)$ is constructed similarly.
For every $i \in [n] \setminus \{\lceil \frac{n+1}{2} \rceil\}$, we replace each vertex $u_i$ with two vertices $v_i < w_i$.
For $i=\lceil \frac{n+1}{2} \rceil$ we set $v_i\colonequals u_i$ and $w_i\colonequals u_i$.
Then we place edges $\{v_i,w_j\}$ and $\{w_i,v_j\}$ for every edge $\{u_i,u_j\}$ of~$\cP_n$.
Additionally, we insert two new vertices $v_0<w_0$ to the left of $v_1$ and add edges $\{v_0,w_0\}$, $\{v_0,w_1\}$, and $\{w_0,v_1\}$ into $\cC_{2n+1}$.
The \emph{blocks of $\cC_{2n+1}$} are the sets $\{v_i,w_i\}$ for $i=1,\ldots,n$.
The edge that contains the $(n+2)$th and the $(n+3)$th vertex of $\cC_{2n+1}$ is the \emph{inner edge of $\cC_{2n+1}$} and the edge $\{v_0,w_0\}$ is the \emph{outer edge of $\cC_{2n+1}$}.
See part~b) of Figure~\ref{fig-altCyc}.

\begin{figure}[ht]
\centering
\includegraphics{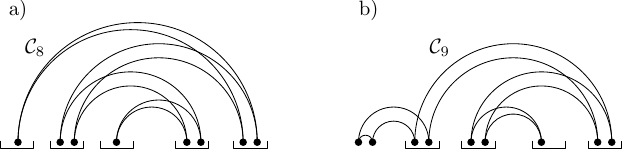}
\caption{a) The alternating cycle $\cC_8$. b) The alternating cycle $C_9$.}
\label{fig-altCyc}
\end{figure}

Using Lemma~\ref{lem-TuranBlowUp}, we can now easily prove Theorem~\ref{thm-evenCycleInBlowUp}. 
We note that no such Tur\'{a}n-type result holds when the given graph is not bipartite.
For example, the ordered graph $\cK_{N/2,N/2}$ contains $N^2/4$ edges, while no ordered odd cycle is an ordered subgraph of $\cK_{N/2,N/2}$.

\begin{proof}[Proof of Theorem~\ref{thm-evenCycleInBlowUp}]
It is sufficient to prove the statement for 2-regular graphs, as every bipartite graph on $n$ vertices with maximum degree 2 is a subgraph of a bipartite 2-regular graph on at most $4n$ vertices.

Let $G$ be a given bipartite 2-regular graph partitioned into even cycles $C_{n_1},\ldots,C_{n_m}$ with $n_1,\ldots,n_m \geq 4$ and $n_1 + \cdots + n_m = n$.
First, we order each cycle $C_{n_i}$ as the alternating cycle~$\cC_{n_i}$.
 The ordering $\cG$ of $G$ is then constructed by placing the vertex set of $\cC_{n_i}$ between the vertices of the inner edge of $\cC_{n_{i-1}}$ for every $i=2,\ldots,m$; see Figure~\ref{fig-altEvenCycEmbed}.

Let $\cH$ be a given ordered graph with $N\geq  2^{16}\varepsilon^{-11} n$ vertices and with at least $\varepsilon N^2$ edges.
We show that $\cG$ is an ordered subgraph of $\cH$.

We apply Lemma~\ref{lem-TuranBlowUp} for $k\colonequals2$, $d \colonequals 2(2/\varepsilon)^2$, and $\cH$ that is partitioned into an order-obeying sequence of intervals $I_1,\ldots,I_{N/d}$, each of size $d$.
Since $N \geq 2d^{2k+1}\varepsilon^{-1}n$, this gives us a copy of $\cP^2_n$ in $\cH$.
To finish the proof, we observe that $\cG$ is an ordered subgraph of $\cP^2_n$.
It suffices to greedily map the vertices of $\cG$ into blocks of $\cP^2_n$ such that every block of $\cC_{n_i}$ is contained in a block of~$\cP^2_n$.
Again, see Figure~\ref{fig-altEvenCycEmbed} for an illustration.
Note that we might not use all blocks of $\cP^2_n$ in the process.
\end{proof}

\begin{figure}[ht]
\centering
\includegraphics{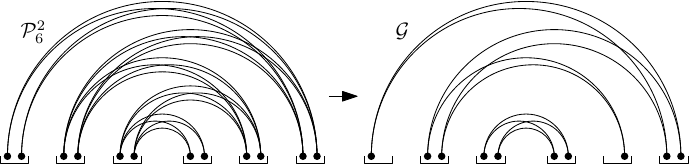}
\caption{The ordering $\cG$ of a 2-regular graph consisting of $\cC_6$ and $\cC_4$ is an ordered subgraph of $\cP^2_6$.}
\label{fig-altEvenCycEmbed}
\end{figure}

For the rest of the section, we deal with the case when $G$ is not bipartite.
To do so, we first introduce an auxiliary ordered graph.

For integers $k$ and $n$, let $(B_1,\ldots,B_n)$ be a sequence of order-obeying intervals, each of size $k$.
Consider the ordered matching $\cM_{2n}=(M_{2n},\prec)$ and let $u_1\prec \cdots \prec u_n$ and $v_1 \prec \cdots \prec v_n$ be the vertices of the left and the right color class of~$\cM_{2n}$, respectively.
We place $\cM_{2n}$ to the left of $B_1$.
For every $i \in [n]$, we then add the edges $\{u_i,w\}$ and $\{v_{n+1-i},w\}$ for every vertex $w \in B_i$ and use $\cT^k_n$ to denote the resulting ordered graph on $(k+2)n$ vertices; see part~a) of Figure~\ref{fig-prefixGraph}.

The intervals $B_1,\ldots,B_n$ are the \emph{blocks of $\cT^k_n$}.
Let $\cH$ be an ordered graph with the vertex set partitioned into a sequence of order-obeying intervals $I_1,\ldots,I_m$.
We say that $\cT^k_n$ is an \emph{ordered subgraph of $\cH$ respecting the partitioning $I_1,\ldots,I_m$} if $\cT^k_n \subseteq \cH$, every block of $\cT^k_n$ is contained in some interval $I_i$ and no two blocks of $\cT^k_n$ are contained in the same interval $I_i$.

\begin{figure}[ht]
\centering
\includegraphics{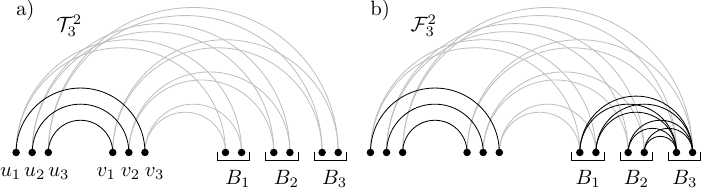}
\caption{a) The ordered graph $\cT^2_3$. b) The ordered graph $\cF^2_3$.}
\label{fig-prefixGraph}
\end{figure}

\begin{lemma}
\label{lem-prefixGraph}
Let $\varepsilon>0$ be a real constant and let $k$ and $d\geq 8\varepsilon^{-2}k$ be positive integers.
Then, for every integer $n$, every ordered graph $\cH$ with $N\geq 2^{14} \varepsilon^{-8}d n$ vertices partitioned into a sequence of order-obeying intervals $I_1,\ldots,I_{N/d}$, each of size $d$, and with at least $\varepsilon N^3$ copies of $\cK_3$ contains $\cT^k_n$ as an ordered subgraph respecting the partitioning $I_1,\ldots,I_{N/d}$.
\end{lemma}
\begin{proof}
We assume that the set $[N]$ is the vertex set of $\cH$.
A copy of $\cK_3$ is called a \emph{triangle}.
Let $u<v<w \in [N]$ be the vertices of a triangle $\cK$ in $\cH$.
The vertex $u$ is the \emph{leftmost} vertex of $\cK$ and $w$ is the \emph{rightmost} vertex of $\cK$.
We call the edges $\{u,v\}$ and $\{v,w\}$ the \emph{left leg} and the \emph{right leg} of $\cK$, respectively.
The \emph{length} of an edge $\{u,v\}$ of $\cH$ equals $|u-v|$.

First, there is a set $T$ of at least $\varepsilon N^3/2$ triangles from $\cH$ with the right leg of length at least $\varepsilon N /2$.
Otherwise the total number of copies of $\cK_3$ in $\cH$ is less than
\[\frac{\varepsilon N^3}{2} + \frac{\varepsilon N^2}{2}\cdot N= \varepsilon N^3,\]
as there are at most $\varepsilon N^2/2$ edges of length smaller than $\varepsilon N/2$ in $\cH$ and for each such edge there are at most $N$ choices for the leftmost vertex of a triangle in $\cH$.

For $i=2,\ldots,N-1$, let $T_i$ be the set of triangles from $T$ with the rightmost vertex in $\{i+1,\ldots,N\}$ and with the remaining two vertices in $\{1,\ldots,i\}$.
Then we have
\[\frac{\varepsilon^2N^4}{4}\leq \frac{\varepsilon N}{2} |T| \leq |\{(i,\cK)\colon \cK \in T_i, i=2,\ldots,N-1)\}|=\sum_{i=2}^{N-1}|T_i|,\]
where the second inequality follows from the fact that every $\cK \in T$ is contained in at least $\varepsilon N/2$ sets $T_i$.
It follows that there is a set $T_j$ of size at least $\varepsilon^2N^3/4$.

We let $A \colonequals \{1,\ldots,j\}$ and $B \colonequals \{j+1,\ldots,N\}$.
Every left leg of a triangle from $T_j$ is an edge of $\cH[A]$.
There is a set $S$ of at least $\varepsilon^2N^2/8$ edges of $\cH[A]$ such each edge from $S$ is the left leg of at least $\varepsilon^2N/4$ triangles from $T_j$.
Otherwise the total number of triangles in $T_j$ is less than
\[\frac{\varepsilon^2N^2}{8}\cdot N + \frac{N^2}{2}\cdot\frac{\varepsilon^2 N}{4} = \frac{\varepsilon^2N^3}{4}.\]

Let $\cR$ be the ordered graph with the vertex set $A$ and with the edge set~$S$.
We know that $\cR$ contains $|A|$ vertices and at least $\varepsilon^2N^2/8 \geq \varepsilon^2|A|^2/8$ edges.
Corollary~\ref{cor-rainbowMatch} implies that $\cR$ contains a copy $\cM$ of $\cM_{2N'}$ for an integer $N' \colonequals \varepsilon^2|A|/16$.
Since $\binom{|A|}{2} \geq \varepsilon^2N^2/8$, we obtain $|A| \geq \varepsilon N/2$ and $N'\geq \varepsilon^3N/2^5$.

We use $u_1 < \cdots < u_{N'}$ and $v_1 < \cdots < v_{N'}$ to denote the vertices of the left and the right color class of $\cM$, respectively.
Let $B'$ be the set $\{d\cdot i \colon i \in [N/d], I_i \cap B \neq \emptyset\}$.
The set $B'$ contains at most $\lceil|B|/d\rceil \leq N/d$ elements and is disjoint with $A$.

Let $\cR'$ be the ordered graph with the vertex set $\{v_1,\ldots,v_{N'}\} \cup B'$ and place an edge $\{v_i,d\cdot l\}$ in $\cR'$ if there are at least $k$ triangles in $\cH$ with the vertex set $\{u_i,v_i,w\}$ for some $w \in I_l\cap B$.

Since every edge $\{u_i,v_i\}$ of $\cM$ is the left leg of at least $\varepsilon^2N/4$ triangles from~$T_j$, every vertex $v_i$ of $\cR'$ has degree at least 
\[\left(\frac{\varepsilon^2 N}{4}-k|B'|\right)\frac{1}{d} \geq \left(\frac{\varepsilon^2 N}{4}-\frac{kN}{d}\right)\frac{1}{d}\geq \frac{\varepsilon^2N}{8d}\]
 in $\cR'$.
The last inequality follows from $d \geq 8\varepsilon^{-2}k$.

The number $|V(\cR')|$ of vertices of $\cR'$ trivially satisfies $\varepsilon^3N/2^5\leq N'\leq |V(\cR')| \leq N$.
Thus $\cR'$ contains at least 
\[\frac{\varepsilon^3N}{2^5}\cdot \frac{\varepsilon^2 N}{8d} = \frac{\varepsilon^5 N^2}{2^8d}\geq \frac{\varepsilon^5}{2^8d}|V(\cR')|^2\] edges.
By Corollary~\ref{cor-rainbowMatch}, there is a copy $\cM'$ of $\cM_{2n}$ in $\cR'$, since $|V(\cR')| \geq \varepsilon^3 N/2^5 \geq 2n/(\varepsilon^5/2^8d)$. 
It follows from the constructions of $\cR$ and $\cR'$ that there is a copy of $\cT^k_n$ in $\cH$ as an ordered subgraph respecting the partitioning $I_1,\ldots,I_{N/d}$.
\end{proof}

Let $k$ and $n$ be positive integers.
By combining the ordered graphs $\cP^k_n$ and $\cT^k_n$, we obtain an ordered graph that is used later to embed orderings of  2-regular graphs.

Let $B_1,\ldots,B_n$ be the blocks of $\cT^k_n$ and let $\cP^k_n$ be the $k$-blow-up of $\cP_n$ with blocks $B_1,\ldots,B_n$.
The ordered graph $\cF^k_n$ is the union of $\cT^k_n$ and $\cP^k_n$; see part~b) of Figure~\ref{fig-prefixGraph}.
The intervals $B_1,\ldots,B_n$ are the \emph{blocks of $\cF^k_n$}.
Observe that we have $\cF^k_n \subseteq \cF^{k'}_{n'}$ for every $k \leq k'$ and $n \leq n'$.

\begin{lemma}
\label{densityTriangles}
For every $\varepsilon > 0$, there is a constant $\delta=\delta(\varepsilon)>0$ such that every ordered graph with $n$ vertices and with at least $(1/4+\varepsilon)n^2$ edges contains at least $\delta n^3$ copies of $\cK_3$. 
\end{lemma}
\begin{proof}
The Triangle Removal Lemma~\cite{ruzSze78} states that for every $\varepsilon' > 0$ there is a $\delta'=\delta'(\varepsilon')>0$ such that every graph on $n$ vertices with at most $\delta'n^3$ copies of $K_3$ can be made $K_3$-free by removing at most $\varepsilon' n^2$ edges.

For a given $\varepsilon>0$, let $\delta>0$ be the parameter $\delta'(\varepsilon)$.
Let $H$ be the underlying graph in the given ordered graph on $n$ vertices with at least $(1/4+\varepsilon)n^2$ edges.
Suppose for contradiction that $H$ contains fewer than $\delta n^3$ copies of $K_3$.
By the Triangle Removal Lemma, we delete fewer than $\varepsilon n^ 2$ edges of $H$ and obtain a $K_3$-free graph with more than $(1/4+\varepsilon)n^2-\varepsilon n^2=n^2/4$ edges.
However, this contradicts Tur\'{a}n's theorem~\cite{turan41}.

Thus $H$ contains at least $\delta n^3$ copies of $K_3$.
Since all orderings of complete graphs are isomorphic, the statement follows.
\end{proof}

Having all auxiliary results, we can now prove Theorem~\ref{thm-2Regular}.

\begin{proof}[Proof of Theorem~\ref{thm-2Regular}]
Again, it is sufficient to prove the statement for 2-regular graphs, as every graph on $n$ vertices with maximum degree 2 is a subgraph of a 2-regular graph on at most $3n$ vertices.

Let $G$ be a given 2-regular graph consisting of cycles $C_{n_1},\ldots,C_{n_m}$ with $n_1,\ldots,n_m \geq 3$ and $n_1 + \cdots + n_m = n$.
First, we order each cycle $C_{n_i}$ as the alternating cycle $\cC_{n_i}$.
For every $i=1,\ldots,m$, let $v^i_1,\ldots,v^i_{n_i}$ be the vertices of $\cC_{n_i}$ in this order.

For $i=2,\ldots,m$, the ordering $\cG$ of $G$ is then constructed iteratively as follows.
For $n_i$ even, we place the vertex set of $\cC_{n_i}$ between the vertices of the inner edge of $\cC_{n_{i-1}}$ (if $n_{i-1}=3$, then we place the vertex set of $\cC_{n_i}$ to the right of $\cC_{n_{i-1}}$).
For $n_i$ odd, we place the outer edge of $\cC_{n_i}$ between the vertices of the outer edge of a previous odd cycle $\cC_{n_j}$.
If $\cC_{n_i}$ is the first odd cycle in the process, then we let $v^i_1$ and $v^i_2$ be the first two vertices in the ordering.
Then we place the vertices $v^i_3,\ldots,v^i_{n_i}$ between the vertices of the inner edge of $\cC_{n_{i-1}}$ (if $n_{i-1}=3$, then we place $v^i_3,\ldots,v^i_{n_i}$ to the right of $\cC_{n_{i-1}}$); see Figure~\ref{fig-2RegOrder}, for an illustration.

\begin{figure}[ht]
\centering
\includegraphics{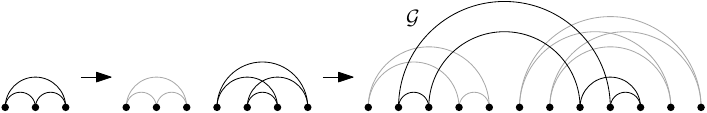}
\caption{The iterative construction of the ordering $\cG$ of a 2-regular graph $G$ with para\-meters $m=3$, $n_1=3$, $n_2=4$, and $n_3=5$.}
\label{fig-2RegOrder}
\end{figure}

The ordering $\cG$ of $G$ is an ordered subgraph of $\cF^2_n$.
For $i=1,\ldots,m$, it suffices to greedily map $v^i_1,\ldots,v^i_{n_i}$ for $n_i$ even and $v^i_3,\ldots,v^i_{n_i}$ for $n_i$ odd into blocks of $\cF^2_n$ such that every block of $\cC_{n_i}$ is contained in a block of $\cF^2_n$.
Then, for every odd $n_i$, we map $v^i_1$ and $v^i_2$ to an edge $\{j,n+1-j\}$ of $\cF^2_n$ such that $j$ and $n+1-j$ are adjacent to vertices of the block of $\cF^2_n$ that contains the first block of $\cC_{n_i}$.
Note that we might not even use all the blocks of $\cF^2_n$ in the process; see Figure~\ref{fig-altCycEmbed} for an illustration.

\begin{figure}[ht]
\centering
\includegraphics{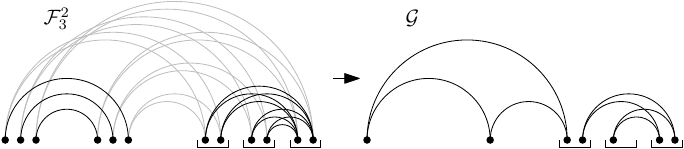}
\caption{The ordered graph $\cF^2_3$ contains the ordering $\cG$ of a 2-regular graph consisting of $\cC_3$ and $\cC_4$.}
\label{fig-altCycEmbed}
\end{figure}

Let $\delta\in(0,1)$ be the parameter $\delta(1/8)$ from Lemma~\ref{densityTriangles}.
Let $N \colonequals 2^{163}\delta^{-10}n$ and let $\chi$ be a coloring of $\cK_N$.
We show that $\chi$ contains a monochromatic copy of $\cF^2_n$.
The rest then follows, as $\cG$ is an ordered subgraph of $\cF^2_n$.

Since $\R(K_3)=6$, every six-tuple of vertices of $\cK_N$ induces a coloring of $\cK_6$ that contains a monochromatic copy of $\cK_3$.
Every such copy is contained in $\binom{N-3}{3}$ six-tuples of vertices of $\cK_N$.
This gives us at least $\binom{N}{6}/\binom{N-3}{3} = \frac{N(N-1)(N-2)}{120}$ monochromatic copies of $\cK_3$ in~$\chi$.
Without loss of generality, we assume that at least half of them are red.
Therefore we have at least $\frac{N(N-1)(N-2)}{240}>N^3/2^8$ red copies of $\cK_3$ in $\chi$. 

We let $N_1 \colonequals 2^{54}\delta^{-8}n$, $\varepsilon_1 \colonequals 1/2^8$, $k_1 \colonequals 2^{12}\delta^{-2}$, and $d_1 \colonequals2^{31}\delta^{-2}$ and apply Lemma~\ref{lem-prefixGraph} with parameters $\varepsilon_1$, $k_1$, and $d_1$ for the ordered subgraph of $\cK_N$ that is formed by red edges in~$\chi$.
Since there are at least $\varepsilon_1 N^3$ red copies of $\cK_3$ in~$\chi$, $d_1 \geq 8\varepsilon_1^{-2}k_1$, and $N \geq 2^{14}\varepsilon_1^{-8}d_1N_1$, this gives a red copy of $\cT^{k_1}_{N_1}$ in $\chi$. 
Let $B_1,\ldots,B_{N_1}$ be the blocks of this red copy of $\cT^{k_1}_{N_1}$ and let $\cH$ be the induced ordered subgraph $\cK_N[B_1 \cup \cdots \cup B_{N_1}]$ of $\cK_N$  with the red-blue coloring induced by $\chi$.
Note that the sets $B_1,\ldots,B_{N_1}$ form an order-obeying sequence of intervals of size $k_1$ that partition the vertex set of $\cH$.
In particular, $\cH$ has $k_1N_1$ vertices.

Now we distinguish two cases.
First, we assume that at least $\binom{k_1N_1}{2}-3k_1^2N_1^2/8>k_1^2N_1^2/2^4$ edges of $\cH$ are red in~$\chi$.
Then we apply Lemma~\ref{lem-TuranBlowUp} with $\varepsilon_2 \colonequals 1/2^4$, $k_2 \colonequals 2$, and $d_2 \colonequals k_1$ for the red ordered subgraph of $\cH$ that is partitioned into $B_1,\ldots,B_{N_1}$.
Since $d_2=k_1 =2^{12}\delta^{-2}> k_2(2/\varepsilon_2)^{k_2}$, $k_1N_1 = 2^{66}\delta^{-10}n > 2d_2^{2k_2+1}\varepsilon_2^{-1} n$, and $\cH$ has at least $\varepsilon_2(k_1N_1)^2$ edges, we obtain a red copy of $\cP^2_n$ as an ordered subgraph of $\cH$ respecting the partitioning $B_1,\ldots,B_{N_1}$. 
Together with the red copy of $\cT^{k_1}_{N_1}$ in $\chi$, this gives a red copy of $\cF^2_n$ in $\cH$.

Otherwise there are more than $3k_1^2N_1^2/8$ blue edges of $\cH$ in $\chi$.
By Lemma~\ref{densityTriangles}, the ordered graph $\cH$ contains at least $\delta k_1^3N_1^3$ blue copies of $\cK_3$.
We apply Lemma~\ref{lem-prefixGraph} with $\varepsilon_3 \colonequals \delta$, $k_3 \colonequals 2^9$, $d_3 \colonequals k_1$, and $N_3 \colonequals 2^{40}n$ for the blue ordered subgraph of $\cH$ that is partitioned into $B_1,\ldots,B_{N_1}$.
Since $d_3=k_1 \geq 8\varepsilon_3^{-2}k_3$ and $k_1N_1 = 2^{66}\delta^{-10} \geq 2^{14}\varepsilon_3^{-8}d_3N_3$, this gives us a blue copy of $\cT^{k_3}_{N_3}$ as an ordered subgraph of $\cH$ respecting the partitioning $B_1,\ldots,B_{N_1}$. 
Let $B'_1,\ldots,B'_{N_3}$ be the blocks of this copy of $\cT^{k_3}_{d_3}$ and let $\cH'$ be the ordered graph with the vertex set $B'_1\cup\cdots\cup B'_{N_3}$ formed by edges of the more frequent color in the coloring $\chi$ of $\cH[B'_1\cup\cdots\cup B'_{N_3}]$.
Again, observe that $B'_1,\ldots,B'_{N_3}$ is an order-obeying sequence of intervals of size $k_3$ that partition the vertex set of $\cH'$.

The ordered graph $\cH'$ has $k_3N_3$ vertices and at least $\binom{k_3N_3}{2}/2>k_3^2N_3^2/8$ edges.
As the last step, we apply Lemma~\ref{lem-TuranBlowUp} with $\varepsilon_4 \colonequals 1/8$, $k_4 \colonequals 2$, $d_4 \colonequals k_3$ for $\cH'$ partitioned into $B'_1,\ldots, B'_{N_3}$.
Since $d_4=k_3 = 2^9 = k_4(2/\varepsilon_4)^{k_4}$ and $k_3N_3 = 2^{49}n\geq 2d_4^{2k_4+1}\varepsilon_4^{-1}n$, we obtain a monochromatic copy of $\cP^2_n$ as an ordered subgraph of $\cH'$ respecting the partitioning $B'_1,\ldots, B'_{N_3}$.  
Together with either the red copy $\cT^{k_1}_{N_1}$ of or with the blue copy of $\cT^{k_3}_{N_3}$, we obtain a monochromatic copy of $\cF^2_n$ in $\chi$.
\end{proof}

\section{A lower bound for ordered matchings with interval chromatic number two}
\label{sec:proofImprovedLowerBoundMatch}

Here we prove Theorem~\ref{thm-improvedLowerBoundMatch}.
That is, we show an improved lower bound on the ordered Ramsey numbers of ordered matchings with interval chromatic number two.

Let $n$ be a positive integer and let $\pi$ be a permutation chosen independently uniformly at random from the set of all permutations of $[n]$.
We recall that the random ordered $n$-matching $\cM(\pi)$ is the ordered matching with the vertex set $[2n]$ and with edges $\{i,n+\pi(i)\}$ for every $i \in [n]$.

\begin{lemma}
\label{lem-matchRandomPermutation}
Let $d$, $n$, $r$, $S$ be positive integers and let $X_1,\ldots,X_d\subseteq [n]$ and $Y_1,\ldots,Y_d \subseteq [2n]\setminus[n]$ be two collections of pairwise disjoint sets such that  $|X_1|\geq \cdots \geq |X_d|$, $|Y_1|\geq \cdots \geq |Y_d|$, and $|X_d||Y_d|\geq S$.
Let $T$ be a set of $r$ pairs $(X_i,Y_j)$ with $1 \leq i,j \leq d$.
Then the probability that we have $e_{\cM(\pi)}(X_i,Y_j)=0$ for every $(X_i,Y_j) \in T$ is less than
\[e^{-\frac{S}{n}\left\lfloor(3d-\sqrt{9d^2-8r})/4\right\rfloor^2}.\]
\end{lemma}
\begin{proof}
We may assume without loss of generality that $|X_1|=\cdots=|X_d|$ and $|Y_1|=\cdots=|Y_d|$, since removing elements from the sets $X_i$ and $Y_j$ does not decrease the probability.
Let $x \colonequals |X_1|=\cdots=|X_d|$ and $y\colonequals |Y_1|=\cdots=|Y_d|$.
From symmetry, we may assume $x \leq y$.

We estimate the probability $P$ that the random $n$-permutation $\pi$ maps no element $e$ from $[n]$ to $\pi(e)$ such that $(e,n+\pi(e)) \in (X_i,Y_j)$ for $(X_i,Y_j) \in T$.
Assume that the elements of~$[n]$ are in some total order $e_1,\ldots,e_n$ in which the values $\pi(e_1),\ldots,\pi(e_n) \in [n]$ are assigned.
The probability $P$ is then at most $\prod_{k=1}^n \min\{\frac{n-|F_{e_k}|}{n-k+1},1\}$, where $F_{e_k} \subseteq [2n]\setminus [n]$ is a set of elements $f$ such that $e_k \in X_i$, $f \in Y_j$ for $(X_i,Y_j) \in T$.

There is a set $Z$ of $z \colonequals \lfloor (3d-\sqrt{9d^2-8r})/4 \rfloor$ indices from~$[d]$ such that for every $i \in Z$ there are at least $2z$ sets $Y_j$ with $(X_i,Y_j) \in T$.
Otherwise $z < \lfloor (3d-\sqrt{9d^2-8r})/4 \rfloor$ and there are fewer than $z d+(d-z)2z \leq r$ pairs in~$T$.
Let $i_1,\ldots, i_z$ be the elements of $Z$.

For every $k=1,\ldots,z$, let $N_k$ be the set of indices $j$ from $[d]$ such that $(X_{i_k},Y_j) \in T$.
Since $\cup_{j\in N_k}Y_j = F_e$ for every $e \in X_{i_k}$, the probability $P$ is at most 
\[\prod_{k=1}^z \left(\frac{n-\sum_{j \in N_k} |Y_j|}{n-\sum_{l = 1}^k|X_{i_l}|}\right)^{|X_{i_k}|} = \prod_{k=1}^z \left(\frac{n-|N_k|y}{n-k x}\right)^x.\]
The denominators are positive, as $x\leq n/d$ and $z<d$.

Since $x \leq y$ and $|N_k| \geq 2z$ for every $k \in [z]$, the last term is at most
\[\left(\frac{n-2z y}{n-z y}\right)^{zx} \leq \left(\frac{n-z y}{n}\right)^{zx} 
= \left(1 - \frac{z y}{n}\right)^{zx} < e^{-xyz^2/n}.\]
Since $xy \geq S$, the probability $P$ is less than
\[e^{-S z^2/n}= e^{-\frac{S}{n}\left\lfloor (3d-\sqrt{9d^2-8r})/4\right\rfloor^2}.\qedhere\]
\end{proof}

In the rest of the section, we set $d \colonequals 3\log{n}$, $S \colonequals 2\cdot 10^4n$, and $r\colonequals\log^2{n}/4$.

\begin{lemma}
\label{lem-matchLargeIntervals}
The random ordered $n$-matching $\cM(\pi)$ satisfies the following statement asymptotically almost surely: for all collections $I_1,\ldots,I_d \subseteq [n]$ and $J_1,\ldots,J_d \subseteq [2n] \setminus [n]$ of pairwise disjoint intervals satisfying $|I_1|\geq \cdots \geq |I_d|$, $|J_1|\geq \cdots \geq |J_d|$, and $|I_d||J_d| \geq S$, the number of pairs $(I_i,J_j)$ with $1\leq i,j\leq d$ and $e_{\cM(\pi)}(I_i,J_j)>0$ is larger than $d^2-r$.
\end{lemma}
\begin{proof}
Let $\cM(\pi)$ be the random ordered $n$-matching.
Let $X$ be a random variable expressing the number of collections $I_1,\ldots,I_d$ and $J_1,\ldots,J_d$ of intervals from the statement of the lemma with $r$ pairs $(I_i,J_j)$, $1\leq i,j \leq d$, that satisfy $e_{\cM(\pi)}(I_i,J_j)=0$.
We show that the expected value of $X$ tends to zero as $n$ goes to infinity.
The rest then easily follows from Markov's inequality.

The number of collections $I_1,\ldots,I_d\subseteq [n]$ and $J_1,\ldots,J_d \subseteq [2n] \setminus [n]$ of pairwise disjoint intervals is at most
\[\binom{n+2d}{2d}^2 \leq \left(\frac{e(n+2d)}{2d}\right)^{4d} < 2^{4d\log{n}}.\]

The number of choices of $r$ elements from a set of $d^2$ elements is
\[\binom{d^2}{r} \leq\left(\frac{ed^2}{r}\right)^r < 2^{7r},\]
where the last inequality follows from the expression of $d$ and $r$.

We fix collections $I_1,\ldots,I_d$ and $J_1,\ldots,J_d$ and $r$ pairs $(I_i,J_j)$ with $1\leq i,j \leq d$.
By Lemma~\ref{lem-matchRandomPermutation}, the probability that $\cM(\pi)$ has no edge between $I_i$ and $J_j$, where $(I_i,J_j)$ is among $r$ chosen pairs, is less than $e^{-\frac{S}{n}\left\lfloor(3d-\sqrt{9d^2-8r})/4\right\rfloor^2}$.

Altogether, the expected value of $X$ is bounded from above by 
\[2^{4d\log{n}}\cdot 2^{7r} \cdot e^{-\frac{S}{n}\left\lfloor(3d-\sqrt{9d^2-8r})/4\right\rfloor^2} < 2^{12\log^2{n}+\frac{7}{4}\log^2{n}- \frac{60}{4}\log^2{n}} =2^{-\frac{5}{4}\log^2{n}}.\]
Thus the expected value of $X$ tends to zero as $n$ goes to infinity, which concludes the proof.
\end{proof}

In the rest of the section, we set $M \colonequals \frac{n\log{\log{n}}}{8\log{n}}$, $s \colonequals \frac{n}{8\log{n}}$, and $t \colonequals \frac{n}{20\log{n}}$.

\begin{lemma}
\label{lem-matchSmallIntervals}
The random ordered $n$-matching $\cM(\pi)$ satisfies the following statement asympto\-tically almost surely: for every $k$ with $1 \leq k \leq t$ and for all partitions $I_1,\ldots,I_k$ of $[n]$ and $J_k,\ldots,J_t$ of $[2n]\setminus[n]$ into an order-obeying sequence of intervals of sizes at most $s$ such that $|I_{i_1}|\geq \cdots \geq |I_{i_k}|$, $|J_{j_1}| \geq \cdots \geq |J_{j_{t-k+1}}|$, and $|I_{i_{d+1}}||J_{j_{d+1}}| < S$, there are more than $M$ pairs $(I_{i_l},J_{j_{l'}})$ with $l,l'>d$ and $e_{\cM(\pi)}(I_{i_l},J_{j_{l'}})>0$.
\end{lemma}

We note that some of the intervals from $I_1,\ldots,I_k,J_k,\ldots,J_t$ might be empty.
Note that it follows from the choice of $s$ that the number of intervals $I_1,\dots,I_k$ is at least $8\log{n}$ and, similarly, the number of intervals $J_k,\dots,J_t$ is at least $8\log{n}$.

\begin{proof}
We assume that $n$ is sufficiently large.
The number of partitions $I_1,\ldots,I_k,\allowbreak J_k,\ldots,J_t$ from the statement is at most
\[\binom{2n+t}{t} \leq \left(\frac{e(2n+t)}{t}\right)^t < (120\log{n})^{n/(20\log{n})} < 2^{n\log{\log{n}}/(19\log{n})}.\]

For such fixed partition $I_1,\ldots,I_k,J_k,\ldots,J_t$ of $[2n]$, the number of choices of $M$ pairs $(I_{i_l},J_{j_{l'}})$ with $l,l'>d$ is at most
\[\binom{t^2}{M} \leq \left(\frac{et^2}{M}\right)^M < 2^{n\log{\log{n}}/8}.\]

Since every interval of the partition has size at most $s$, there are at most $2ds$ edges with one endpoint in $I_{i_1}\cup\cdots\cup I_{i_d}\cup J_{j_1}\cup \cdots\cup J_{j_d}$.
Thus the number of edges of $\cM(\pi)$ between intervals $I_{i_l}$ and $J_{j_{l'}}$ with $l,l' > d$ is at least $n/4$, as we have $n-2ds = n - 3n/4 = n/4$.

Let $P$ be the probability that for a fixed partition $I_1,\ldots,I_k,J_k,\ldots,J_t$ and for a fixed set $T$ of $M$ pairs $(I_{i_l},J_{j_{l'}})$ with $l,l'>d$ the random $n$-matching $\cM(\pi)$ satisfies $e_{\cM(\pi)}(I_{i_l},J_{j_{l'}})=0$ for every pair $(I_{i_l},J_{j_{l'}})$ with $l,l'>d$ that is not contained in $T$.
We show that the probability $P$ is less than
\[\max_{\alpha \in [1/4,1]}{\frac{\binom{MS}{\alpha n}}{(\alpha n)!}}.\]

We fix an ordered matching $\cM$ that is formed by edges incident to a vertex in $I_{i_1}\cup\cdots\cup I_{i_d}\cup J_{j_1}\cup \cdots\cup J_{j_d}$.
Let $n_{\cM}$ be the number of edges in $\cM$ and let $P_{\cM}$ be the probability that $\cM$ is in $\cM(\pi)$.
Then there are $n-n_{\cM} \geq n/4$ edges of $\cM(\pi)$ that are not in $\cM$.
These edges are contained in pairs from $T$ with probability less than $\binom{MS}{n-n_{\cM}}/((n-n_{\cM})!)$, as $|I_i||J_j| < S$ for every pair $(I_i,J_j) \in T$.
Taking the maximum of $\binom{MS}{\alpha n}/(\alpha n)!$ over $\alpha \in [1/4,1]$, we have
\[P < \sum_{\cM} P_{\cM}\frac{\binom{MS}{n - n_{\cM}}}{(n - n_{\cM})!} \leq \max_{\alpha \in [1/4,1]}{\frac{\binom{MS}{\alpha n}}{(\alpha n)!}}\sum_{\cM}P_{\cM} = \max_{\alpha \in [1/4,1]}{\frac{\binom{MS}{\alpha n}}{(\alpha n)!}},\]
where the summation goes over all ordered matchings $\cM$ that are formed by edges incident to a vertex in $I_{i_1}\cup\cdots\cup I_{i_d}\cup J_{j_1}\cup \cdots\cup J_{j_d}$.

Using the standard estimates $(a/e)^a\leq a!$ and $\binom{a}{b}\leq\left(\frac{ea}{b}\right)^b$ for positive integers $a$ and $b$, we bound $\binom{MS}{\alpha n}/(\alpha n)!$ from above by
\begin{align*}
\left(\frac{eMS/(\alpha n)}{\alpha n/e}\right)^{\alpha n} &= \left(\frac{e^2MS}{(\alpha n)^2}\right)^{\alpha n}=\left(\frac{2\cdot 10^4e^2 \log{\log{n}}}{8\alpha^2\log{n}}\right)^{\alpha n} \\
&< (\log{n})^{-4\alpha n/7}\leq 2^{-n\log{\log{n}}/7}
\end{align*}
for a sufficiently large $n$.

Let $X$ be the random variable expressing the number of partitions $I_1,\ldots,I_k,\allowbreak J_k,\ldots,J_t$ of $[2n]$ from the statement of the lemma such that the number of pairs $(I_{i_l},J_{j_{l'}})$ satisfying $l,l'>d$ and $e_{\cM(\pi)}(I_{i_l},J_{j_{l'}})>0$ is at most $M$.
It follows from our observations that, for a sufficiently large $n$, the expected value of $X$ is at most
\[2^{n\log{\log{n}}/(19\log{n})} \cdot 2^{n\log{\log{n}}/8} \cdot 2^{-n\log{\log{n}}/7} .\]
Thus we see that the expected value of $X$ tends to zero as $n$ goes to infinity.
The rest of the statement then follows from Markov's inequality.
\end{proof}

We now prove Theorem~\ref{thm-improvedLowerBoundMatch}.
The main idea of the proof is similar to the one used by Conlon et al.~\cite{cfls14} in the proof of Theorem~\ref{thm-lowerBoundMatch}.

\begin{proof}[Proof of Theorem~\ref{thm-improvedLowerBoundMatch}]
Let $\cM$ be an ordered matching of interval chromatic number two and with the color classes of size $n$ that satisfies the statements from Lemma~\ref{lem-matchLargeIntervals} and Lemma~\ref{lem-matchSmallIntervals}.
We know that the random $n$-matching satisfies both statements asymptotically almost surely.

Let $\cR$ be the ordered complete graph with loops on the vertex set $[t]$.
Let $\chi$ be the coloring of $\cR$ that assigns either a red or a blue color to each edge of $\cR$ independently at random with probability $1/2$.
Let $A_1,\ldots,A_t$ be the partition of the vertex set of $\cK_{st}$ into an order-obeying sequence of $t$ intervals of size $s$.
We define a coloring $\chi'$ of $\cK_{st}$ as follows.
The color $\chi'(e)$ of an edge $e$ of $\cK_{st}$  is $\chi(\{i,j\})$ if one endvertex of $e$ is in $A_i$ and the other one is in $A_j$.

We show that the probability that there is a red copy of $\cM$ in $\chi'$ is less than $1/2$.
Let $\cM_0$ be a red copy of $\cM$ in $\chi'$ and let $I$ and $J$ be the color classes of $\cM_0$ in this order.

For $i=1,\ldots,t$, we set $I_i \colonequals A_i \cap I$.
Let $k\leq t$ be the maximum $i$ such that $I_i$ is nonempty.
For $j=k,\ldots,t$, we set $J_j \colonequals A_j \cap J$.
This gives us partitions $I_1,\ldots,I_k$ and $J_k,\ldots,J_t$ of $I$ and $J$, respectively, into order-obeying sequences of intervals, each of size at most $s$.
We consider the orderings $I_{i_1},\ldots,I_{i_k}$ and $J_{j_1},\ldots,J_{j_{t-k+1}}$ of $I_1,\ldots,I_k$ and $J_k,\ldots,J_t$, respectively, such that $|I_{i_1}|\geq \cdots \geq |I_{i_k}|$ and $|J_{j_1}| \geq \cdots \geq |J_{j_{t-k+1}}|$.

First, we assume that $|I_{i_d}||J_{j_d}| \geq S$.
Since $\cM$ satisfies the statement from Lemma~\ref{lem-matchLargeIntervals}, there are at least $d^2-r$ pairs $(I_i,J_j)$ with $i \in \{i_1,\ldots,i_d\}$, $j \in \{j_1,\ldots,j_d\}$, and $e_{\cM_0}(I_i,J_j)>0$.
From the choice of $\chi'$, the red copy $\cM_0$ thus corresponds to a red ordered subgraph of $\cR$ in $\chi$ with $2d$ or $2d-1$ vertices and with at least $d^2-r$ edges, one of which might be a loop.
By the union bound, the probability that there is such an ordered graph in $\cR$ is at most 
\begin{align*}
\left(\binom{t}{2d}+\binom{t}{2d-1}\right)\binom{d^2}{r}&2^{-d^2+r} \leq 2\left(\frac{et}{2d}\right)^{2d}\left(\frac{ed^2}{r}\right)^{r}2^{-d^2+r}\\
&< 2^{6\log^2{n}}\cdot 2^{2\log^2{n}} \cdot 2^{-9\log^2{n}+\log^2{n}/4}=2^{-\frac{3}{4}\log^2{n}},
\end{align*}
where the last inequality follows from the choice $d=3\log{n}$, $t=\frac{n}{20\log{n}}$, and $r=\log^2{n}/4$.
For a sufficiently large $n$, this expression is less than $1/4$.

Now, we assume that $|I_{i_d}||J_{j_d}| < S$.
The matching $\cM$ satisfies the statement from Lemma~\ref{lem-matchSmallIntervals} and thus there are more than $M$ pairs $(I_{i_l},J_{j_{l'}})$ with $l,l'>d$ and $e_{\cM(\pi)}(I_{i_l},J_{j_{l'}})>0$.
From the choice of $\chi'$, the red copy $\cM_0$ corresponds to a red ordered subgraph of $\cR$ in $\chi$ with at least $M$ edges that are determined by the partition $I_1,\ldots,I_k,J_k,\ldots,J_t$ of $[2n]$.
By the union bound, the probability that there is such ordered graph in $\chi$ is at most
\[\binom{2n+t}{t}2^{-M} \leq \left(\frac{e(2n+t)}{t}\right)^t 2^{-M} < 2^{\frac{n\log{\log{n}}}{19\log{n}}}\cdot 2^{-\frac{n\log{\log{n}}}{8\log{n}}} = 2^{-\frac{11n\log{\log{n}}}{152\log{n}}}.\]
If $n$ is sufficiently large, then this term is less than $1/4$.

In total, the probability that there is a red copy of $\cM$ in $\chi'$ is less than $1/4+1/4=1/2$.
From symmetry, a blue copy of $\cM$ appears in $\chi'$ with probability less than $1/2$.
Altogether, the probability that there is no monochromatic copy of $\cM$ in $\chi'$ is positive.
In other words, we have $\OR(\cM) \geq  st = \frac{1}{160}\left(\frac{n}{\log{n}}\right)^2$.
\end{proof}

\section*{Acknowledgments}
The authors were supported by the grant GA\v{C}R 14-14179S.
The first author acknowledges the support of the Grant Agency of the Charles University, GAUK 690214 and the project SVV-2015-260223. The second author received financial support from the Neuron Foundation for Support of Science.

We would like to thank the anonymous referees for their comments that helped to improve the presentation of the paper.


\begin{thebibliography}{00}


\bibitem{bckk13}
M. Balko, J. Cibulka, K. Kr\'{a}l, J. Kyn\v{c}l.
Ramsey numbers of ordered graphs, 
submitted. Preprint available at
\href{http://arxiv.org/abs/1310.7208}{\tt{arXiv:1310.7208}}, 2013.

\bibitem{bckk15}
M. Balko, J. Cibulka, K. Kr\'{a}l, J. Kyn\v{c}l.
Ramsey numbers of ordered graphs, 
\emph{Electron. Notes Discrete Math.}, \textbf{49}, 419--424, 2015.

\bibitem{benCan78}
E.A. Bender, E.R. Canfield.
The asymptotic number of labeled graphs with given degree sequences,
\emph{J. Combin. Theory Ser. A} \textbf{24}(3), 296--307, 1978.

\bibitem{crst83}
V. Chv\'{a}tal, V. R\"{o}dl, E. Szemer\'{e}di, W.T. Trotter Jr.
The Ramsey number of a graph with bounded maximum degree,
\emph{J. Combin. Theory Ser. B} \textbf{34}(3), 239--243, 1983.

\bibitem{cfls14}
D. Conlon, J. Fox, C. Lee, B. Sudakov.
Ordered Ramsey numbers, 
\emph{J. Combin. Theory Ser. B} \textbf{122}, 353--383, 2017.

\bibitem{cfs15}
D. Conlon, J. Fox, B. Sudakov.
Recent developments in graph Ramsey theory, 
\emph{Surveys in Combinatorics}, 49--118, 2015.

\bibitem{erdos47}
P. Erd\H{o}s.
Some remarks on the theory of graphs,
\emph{Bull. Amer. Math. Soc.} \textbf{53}(4), 292--294, 1947.

\bibitem{kovari54}
T. K\H{o}v\'{a}ri, V.T. S\'{o}s, P. Tur\'{a}n.
On a problem of K. Zarankiewicz,
\emph{Colloq. Math.} \textbf{3}, 50--57, 1954.

\bibitem{mat99}
J. Matou\v{s}ek.
Geometric discrepancy. An illustrated guide,
Algorithms and Combinatorics, 18. Springer-Verlag, Berlin. ISBN 3-540-65528-X, 1999.

\bibitem{rams30}
F.P. Ramsey.
On a Problem of Formal Logic,
\emph{Proc. London Math. Soc.} \textbf{S2-30}(1), 264--286, 1930.

\bibitem{ruzSze78}
I.Z. Ruzsa, E. Szemer\'{e}di.
Triple systems with no six points carrying three triangles,
\emph{Combinatorics (Keszthely, 1976), Coll. Math. Soc. J. Bolyai} \textbf{18}, Volume II, 939--945, 1978.

\bibitem{turan41}
P. Tur\'{a}n.
On an extremal problem in graph theory,
\emph{Mat. Fiz. Lapok} (in Hungarian) \textbf{48}, 436--452, 1941.

\bibitem{wor78}
N.C. Wormald.
Some problems in the enumeration of labelled graphs,
PhD thesis, University of Newcastle, 1978.

\end{thebibliography}
\end{document}